\documentclass{amsart}
\usepackage{amssymb}
\usepackage{amsthm}
\usepackage{epsfig,color}
\usepackage{url}
\usepackage{picinpar}
\usepackage[all]{xy}
\newcommand{\cd}[1]{\begin{equation*}{\xymatrix{#1}}\end{equation*}}

\def\co{\colon\thinspace}

\def\Z{\mathbb{Z}}

\newtheorem{theorem}{Theorem}[section]
\newtheorem{lemma}[theorem]{Lemma}
\newtheorem{corollary}[theorem]{Corollary}
\newtheorem{proposition}[theorem]{Proposition}
\newtheorem{question}[theorem]{Question}

\theoremstyle{definition}
\newtheorem{remark}[theorem]{Remark}
\newtheorem{definition}[theorem]{Definition}
\newtheorem{example}[theorem]{Example}

\newcommand{\underw}{{\underline{w}}}
\newcommand{\onto}{\twoheadrightarrow}
\newcommand{\scl}{\ensuremath{\mathrm{scl}}}

\begin{document}

\date{1 November 2009}

\author[J. F. Manning]{Jason Fox Manning}
\address{University at Buffalo, SUNY}
\email{j399m@buffalo.edu}

\title{Virtually geometric words and Whitehead's algorithm}

\begin{abstract}
Motivated by a question of Gordon and Wilton, we consider the question
of which collections of words are ``virtually geometric''.  In
particular, we prove that some words (e.g. $bbaaccabc$) are not
virtually geometric.
\end{abstract}

\maketitle

\section{Definitions and introduction}
Let $G$ be a group, $\underline{g}=\{[g_i]\}_{i\in I}$ a collection of conjugacy classes
in $G$, and $X$ some space.  Any homomorphism $\phi\co G\to \pi_1(X)$
induces a map, 
\[ \phi_{\underline{g}}\co \bigsqcup_{i\in I}{S^1}_i\to X, \]
well-defined up to homotopy, which sends the circle ${S^1}_i$ to a
loop in $X$ freely homotopic to the element $\phi(g_i)$.

\begin{definition}\label{d:geometric}
Let $F_n=\langle x_1,\ldots,x_n\rangle$ be the free group of rank $n$,
and let $H$ be a $3$--dimensional
handlebody of genus $n$.
We say a collection of conjugacy classes 
$\underw=\{[w_1],\ldots,[w_r]\}$ in $F$ is
\emph{geometric} if for some isomorphism 
\[ \phi\co F_n\to \pi_1(H) \]
the induced map $\phi_{\underw}$ is homotopic in $H$ to an
embedding 
\[\phi_{\underw}'\co \bigsqcup_{i=1}^r{S^1}_i\hookrightarrow \partial H. \]
In this case we also say that the finite presentation
\[ \langle x_1,\ldots,x_n\mid w_1,\ldots w_r\rangle \]
is \emph{geometric}.
\end{definition}
\begin{remark}
  Let $\phi\co F_n\to \pi_1(H)$ be an isomorphism as in Definition \ref{d:geometric}.
  Since any automorphism of a free group $F_n$ can be realized as
  a homeomorphism of the handlebody $H$, the phrase ``for some
  isomorphism'' can be replaced by ``for any isomorphism''.  (To see
  that any automorphism can be realized, first realize elementary
  Nielsen transformations 
\[\alpha_i\co \quad x_i\mapsto x_i^{-1}\mbox{, }x_j\mapsto
  x_j,\forall j\neq i,\] 
  and \[\beta_{ij} \co \quad x_i \mapsto x_ix_j\mbox{, }x_k\mapsto x_k,\forall
  k\neq i,\] and
  then use the fact that these transformations generate
  $\mathrm{Aut}(F_n)$, \cite[I.4]{LS}.)
\end{remark}
\begin{remark}
  A geometric presentation $\langle x_1,\ldots,x_n\mid w_1,\ldots
  w_r\rangle$ gives a (not necessarily 
  unique) way to build a $3$--manifold $M$ with $\pi_1(M) = \langle x_1,\ldots,x_n\mid w_1,\ldots
  w_r\rangle$, 
  by attaching $2$--handles to
  the boundary of the handlebody $H$, along attaching curves 
  $\phi_{\underw}'({S^1}_i)$ for $1\leq i\leq r$.  Thus every
  geometric presentation 
  is a presentation of the fundamental group of some $3$--manifold. 

  Moreover, the \emph{double} $D_n(\underw)$
  \[ \left.\left\langle \begin{array}{c}
                  x_1,\ldots,x_n,\\
                  x_1',\ldots,x_n',\\
                  t_2,\ldots,t_r\end{array}\right|
  \begin{array}{c}
    w_1(x_1,\ldots,x_n)=w_1(x_1',\ldots,x_n'),\\
    t_i(w_i(x_1,\ldots,x_n))t_i^{-1}=w_i(x_1',\ldots,x_n'), i\geq 2
  \end{array}
  \right\rangle \]
  is the fundamental group of a $3$--manifold built by 
  attaching thickened annuli to two copies of the handlebody $H$
  of genus $n$, along the attaching curves given by $\phi_\underw'$.
\end{remark}
\begin{definition}
  Let $\underw=\{[w_1],\ldots,[w_r]\}$ be a collection of
  conjugacy classes in the free group $F = F_n$.  For $F'<F$ a finite
  index subgroup, we define the \emph{lifts of $\underw$} to be the
  following set of $F'$--conjugacy classes:
\[\underw_{F'} = \left\{[gw^{n(g)}g^{-1}]\ \left|\ %
\begin{array}{c}  g\in
F,\ w\in\underw,\mbox{ and }
\\n(g)\in\Z_{>0} \mbox{ minimal so that }
gw^{n(g)}g^{-1}\in F'
\end{array}
\right\}\right. \]
  The collection $\underw$ is \emph{virtually
    geometric} if $\underw_{F'}$ is geometric for some finite
  index $F'<F$.  In this case, we say the presentation 
\[ \langle x_1,\ldots,x_n\mid w_1,\ldots w_r \rangle \]
  is \emph{virtually geometric}.
\end{definition}
Here is an equivalent topological
formulation:  Realize $\underw$ as an embedded
collection of circles $N_{\underw}$ in $H$.  If, for some finite
cover $\tilde{H}\stackrel{p}{\to} H$, the inclusion
\[\tilde\iota\co p^{-1}(N_{\underw})\to \tilde{H} \]
is homotopic to an embedding into 
$\partial\tilde{H}$, then $\underw$ is virtually geometric.
\begin{remark}\label{r:vg3mfd}
  If a presentation $\langle x_1,\ldots,x_n\mid w_1,\ldots w_r
  \rangle$ is virtually geometric, then
   there are positive integers
  $k_1,\ldots,k_r$ so that
 $$ \langle x_1,\ldots,x_n\mid
  w_1^{k_1},\ldots w_r^{k_r} \rangle$$ 
  is virtually a $3$--manifold
  group.
  Indeed, realizing $\underw=\{w_1,\ldots,w_r\}$ as a collection of
  embedded loops $N_{\underw}$ in a genus $n$ handlebody $H$, let
  $\tilde{H}\stackrel{p}{\longrightarrow} H$ be a \emph{regular} finite
  cover in which the preimage $\pi^{-1}(N_{\underw})$ is homotopic to
  an embedding in the boundary.  For $i\in \{1,\ldots,r\}$ let
  $\gamma_i\subset H$ be the loop in $N_{\underw}$ corresponding to
  $w_i$.  Each component of $p^{-1}(\gamma_i)$ covers $\gamma_i$ with
  the same degree;  let $k_i$ be this degree.
  We construct a complex $\tilde{K}$ from $\tilde{H}$ by attaching
  $k_i$ disks to each component of $p^{-1}(\gamma_i)$, for each $i$.
  Since $p^{-1}(N_{\underw})$ is homotopic to an embedding into
  $\partial \tilde{H}$, 
  the complex $\tilde{K}$ is homotopy
  equivalent to a $3$--manifold.  Moreover, the covering map
  $\tilde{H}\to H$ extends to a covering map $\tilde{K}\to K$ where
  $K$ is a complex with   $\pi_1(K) = \langle x_1,\ldots,x_n\mid w_1^{k_1},\ldots w_r^{k_r}
  \rangle$.\footnote{Actually, one can use the Orbifold Theorem
    \cite{BLP,CHK} together with a theorem of McCullough and Miller
    \cite{MM} to show that 
 $$ \langle x_1,\ldots,x_n\mid
  w_1^{N},\ldots w_r^{N} \rangle$$ 
  is a virtual $3$--manifold group for some positive integer $N$, but the author does not know an elementary
    argument.}

  (An aside:  It is possible to show 
  that every finitely
  presented group admits a  presentation which is virtually geometric, but the virtually
  geometric presentation will nearly always be different from the
  original one.  Briefly, one realizes the group as the fundamental
  group of a special polyhedron; a finite branched cover of this polyhedron
  with positive even-degree branching over the center of every face
  will be thickenable.  It follows that a presentation coming from the
  special polyhedron is virtually geometric.
  For more on  \emph{special 
    polyhedra} and \emph{thickenability}, see Chapter 1 of \cite{Matveev}.)

  The double $D_n(\underw)$ along a virtually geometric collection
  $\underw$ is also virtually the fundamental group of a $3$--manifold,
  obtained by joining together a pair of handlebodies by thickened
  annuli.
  A $3$--manifold
  argument based on Dehn's Lemma \cite[Lemma 20]{GWxx}
  then shows that if $D_n(\underw)$ is one-ended, it contains
  a closed surface subgroup.
\end{remark}
Gordon and Wilton gave some examples of virtually geometric but
non-geometric words in \cite{GWxx}, and asked the following:
\begin{question}\cite[Question 22]{GWxx}\label{q:main}
  Let $F$ be a free group of rank $n\geq 2$.  Is $\{[w]\}$ virtually
  geometric for every $w\in F$?
\end{question}
We give a negative answer to this question in Section \ref{s:answer},
by showing that the word $bbaaccabc\in \langle a,b,c\rangle$ is not
virtually geometric (See Corollary \ref{cor:k33}).

\begin{remark}
Danny Calegari (personal communication) pointed out a second motivation for studying
Question \ref{q:main}, connected with \emph{stable commutator length}
\cite{scl}. 
If a 
homologically trivial 
collection $\underw$ is geometric, then $\scl(\underw)$ is in
$\Z+\frac{1}{2}$, by \cite[Proposition 4.4]{scl}.  (This Proposition
is an application of Gabai's result \cite[6.18]{Ga} on the
proportionality of Thurston and Gromov norms on the second homology of
a $3$--manifold.)  
More generally, suppose $\underw$ is virtually geometric, so that
$\underw_{F'}$ is geometric for some $F'$ with $[F:F']<\infty$.
A straightforward argument then shows that
\[\scl(\underw)\in\frac{1}{[F:F']}(\Z+\frac{1}{2}).\]
Thus a positive answer to Question \ref{q:main} for words in $[F,F]$
would have given an
alternate proof of the rationality of $\scl$ in free groups, which is the main
theorem of \cite{calegari:pql}.  
In fact, the answer to Question \ref{q:main} is negative even in this
more restricted setting (see Corollary \ref{cor:comm}).
\end{remark}

\section{Whitehead graphs}\label{s:whitehead}
This section describes certain graphs which arise naturally from
collections of words together with disk systems in handlebodies.  The
material in this section is discussed in more detail in Berge \cite{Berge}.

In the graphs we consider,
multiple edges may connect a pair of vertices, but no edge joins a
vertex to itself.  The \emph{valence} of a vertex is the number of
incident edges (not the number of neighbors), and a graph will be said
to be \emph{regular} if every vertex has the same valence.

A collection of conjugacy classes in a free group
$F = \langle x_1,\ldots, x_n\rangle$ can also be thought
of as a collection of \emph{cyclic words}, i.e., freely and cyclically
reduced ``words'' $w_1,\ldots,w_r$ whose letters are indexed mod $l_i$
where $l_i$ is the length of $w_i$.  Given a collection
$\underw$ of cyclic words, one can form the 
\emph{Whitehead graph} $W(\underw)$ as follows:
\begin{enumerate}
\item The $2n$ vertices of $W(\underw)$ are the generators $x_1,\ldots,x_n$
  of $F$ and their inverses.
\item For each two-letter sequence $xy$ in some $w_i\in \underw$,
  attach an edge from $x^{-1}$ to $y$.  (If some $w_i = x$ has length
  one, attach an edge from $x^{-1}$ to $x$.)
\end{enumerate}
\begin{example}
  If $F$ is the free group $\langle a,b,c\rangle$, and $\underw =
  \{bbaaccabc\}$, then the Whitehead graph $W(\underw)$ is isomorphic
  to the complete bipartite graph $K_{3,3}$.  
\end{example}
\begin{remark}\label{r:wggeo}
Again, there is a topological interpretation:  
Realize the collection
$\underw$ as an embedded $1$--submanifold $N$ in the handlebody $H$.
There is a collection of disks $\mathcal{D}=\{d_1,\ldots,d_n\}$ in $H$
so that for each $i$ there is a loop in $H$ representing 
the generator $x_i$ which
intersects $d_i$ exactly once, and has empty intersection with $d_j$
for all $j\neq i$.  After homotoping $N$ to intersect
$\mathcal{D}$ minimally, split the handlebody along $\mathcal{D}$ to
give a ball with $2n$ disks $\mathcal{D}_{\pm}$ in the boundary, and
$N'$ a $1$--submanifold with boundary in $\mathcal{D}_\pm$.
The Whitehead graph can then be obtained as the quotient of
$\mathcal{D}_\pm\cup N'$ by crushing each disk to a point.
\end{remark}

Notice that if $N$ was homotopic to an embedding into $\partial H$
which still intersects $\mathcal{D}$ minimally,
then the Whitehead graph is planar.
As shown by an example of Berge \cite[Section
  12]{Berge}, the assumption that $N$ still intersects $\mathcal{D}$
minimally after being homotoped into $\partial H$ is necessary, even
for geometric $\underw$;  there are examples of
geometric collections of words with non-planar Whitehead graph.
However, we will see in Theorem \ref{t:planar} that if $\mathcal{D}$
is \emph{minimal} in a certain sense, 
and $\underw$ is geometric, then the Whitehead graph is
planar.  

A \emph{complete system of disks} for $H$ is a collection
$\mathcal{C}$ of properly embedded disks which is maximal with respect
to the condition that $H\smallsetminus \cup\mathcal{C}$ is connected.
The system $\mathcal{D}$ from Remark \ref{r:wggeo} is an example of a
complete system.  Different choices of disks determine
different choices of free generators for $\pi_1(H)$.  (Actually, there
are multiple complete disk systems corresponding to any choice of free
generators; these disk systems are related to one another by
homeomorphisms of $H$ which act trivially on $\pi_1(H)$.)

Any collection of disjoint properly embedded disks in $H$
(complete or not)
gives rise to a 
graph $W(\underw, \mathcal{C})$ in exactly the same manner as
described in Remark \ref{r:wggeo}.  
So long as $\mathcal{C}$ is a complete system, the graph
$W(\underw,\mathcal{C})$ always has $2n$ 
vertices, but the number of edges depends on the particular complete
system $\mathcal{C}$.
(The number of edges is the number of times $\mathcal{C}$ intersects
$\underw$, after $\underw$ has been homotoped to make this
intersection minimal.)
The complete system $\mathcal{C}$ is said to be
\emph{minimal} (with respect to $\underw$) if the number of edges is
minimal over all complete systems of disks for $H$.  

As explained in \cite{Berge}, the following can be deduced from a
result of Zieschang \cite[Theorem 1]{Zi}.
\begin{theorem}\cite[Section 11]{Berge}\label{t:planar}
  If $\mathcal{C}$ is minimal, and $\underw$ is geometric, then
  $W(\underw,\mathcal{C})$ is planar.
\end{theorem}
(In case the hypotheses of Theorem \ref{t:planar} are satisfied,
Berge \cite[Section 13]{Berge} explains further how to obtain a
\emph{Heegaard diagram}, and thus an explicit embedding of
$N_{\underw}$ in $\partial H$, from
$W(\underw,\mathcal{C})$.) 

The next theorem is explained in \cite[Section 8]{Berge}, and follows
from the ``peak reduction'' lemma of Higgins--Lyndon (see \cite{HL}
or \cite[I.4]{LS}, 
cf. \cite[Section 4]{CV} for an account with a different flavor).
\begin{theorem}\cite[Section 8]{Berge}\label{t:reducible}
  If $\mathcal{C}$ is not minimal, then $W(\underw,\mathcal{C})$ contains
  a pair of vertices $v_+$ and $v_-$ with the same valence $k$ and a
  collection of $s<k$ edges $\{e_1,\ldots,e_s\}$ which separate $v_+$
  from $v_-$.
\end{theorem}
\begin{remark}
  The pair of vertices $v_\pm$ in Theorem \ref{t:reducible} is the
  pair of vertices coming from some single disk $D$ in $\mathcal{C}$.  One
  can obtain a new disk system $\mathcal{C}'$
  for $H$ by deleting $D$ and adding a
  disk $D'$ which intersects each of the edges $e_1,\ldots,e_s$
  exactly once, without intersecting any other edges.  The Whitehead
  graph $W(\underw,\mathcal{C}')$ is obtained from
  $W(\underw,\mathcal{C})$ by adding a new pair of vertices
  (corresponding to $D'$), deleting $\{v_\pm\}$, and ``splicing together''
  the edges incident to $\{v_\pm\}$.  (We will give a precise
  definition of \emph{splicing} in the next section.)
  An example of such an operation
  (called a \emph{Whitehead move}) is
  shown in Figure \ref{f:whmove}.
  \begin{figure}[htpb]
    \input{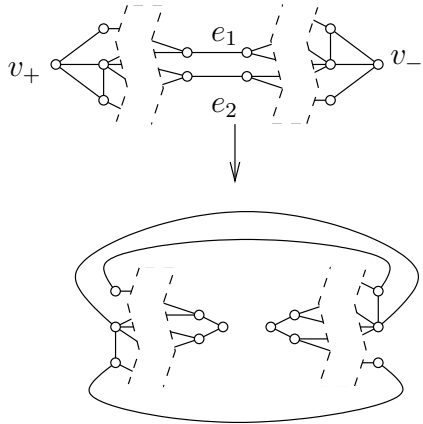}
    \caption{Reducing the number of edges in the Whitehead graph by a
      Whitehead move (an ``unsplice'' followed by a splice).}
    \label{f:whmove}
  \end{figure}
\end{remark}

\section{Splicing}
Let $G_1$ and $G_2$ be graphs containing 
vertices $v_1\in G_1$, $v_2\in
G_2$ of the same valence.  Let
$\sigma$ be a bijection from the edges 
incident to $v_1$ to the edges incident to $v_2$.  We can form a new
graph $G$ from this data by the procedure:
\begin{enumerate}
\item Form the disjoint union of $G_1$ and $G_2$.
\item Delete the vertices $v_1$ and $v_2$ (and any adjacent edges).
\item For each edge $e$ connecting $v_1$ to 
  some $v'\in G_1$, the
  corresponding edge $\sigma(e)$ connects $v_2$ to some $v''\in G_2$.
  Add an edge to $G$ connecting $v'$ to $v''$.
\end{enumerate}
If $G$ is obtained by such a procedure, we say that $G$ is obtained by
\emph{splicing $G_1$ and $G_2$ at $(v_1,v_2)$}.  (There may be more
than one such $G$ for a given $G_1,G_2,(v_1,v_2)$.)
\begin{lemma}\label{l:minor}
  Let $v_i\in G_i$ be vertices for $i\in \{1,2\}$.
  If $G_2\smallsetminus v_2$ is connected, and $G$ is obtained by splicing
  $G_1$ and $G_2$ at $(v_1,v_2)$, then $G_1$ is a minor of $G$.
\end{lemma}
\begin{proof}
  Collapse $G_2\smallsetminus v_2$ to a point to obtain a graph
  isomorphic to $G_1$.
\end{proof}
\begin{definition}
  We say a graph is \emph{$p$--edge-connected} if no collection of
  $p-1$ or fewer edges disconnects the graph.
\end{definition}
\begin{lemma}\label{l:connected}
  Let $k>0$, and suppose that $G$ is obtained by splicing $G_1$ and
  $G_2$ at $(v_1,v_2)$.  If $G_1$ and $G_2$ are regular, $k$--valent,
  $k$--edge-connected graphs, then so is $G$. 
\end{lemma}
\begin{proof}
  For $i\in \{1,2\}$, let $G_i'$ be the graph obtained from $G_i$ by
  removing $v_i$ and all the edges adjacent to $v_i$.  The graph
  $G_i'$  will be thought of either as a subgraph of $G_i$ or as a
  subgraph of $G$, depending on context.

  For $i\in \{1,2\}$ we claim that $G_i'$ is 
  $(\lfloor\frac{k-1}{2}\rfloor+1)$--edge-connected.  
  Indeed, suppose that $G_1'$ is
  disconnected by $\lfloor\frac{k-1}{2}\rfloor$ or fewer
  edges into components
  $C_1,\ldots C_c$ for some $c\geq 2$.  For some $1\leq j\leq c$, the
  number of edges connecting $v_1$ to $C_j$ in $G_1$
  is at most $\lfloor\frac{k}{2}\rfloor$.  It follows that $G_1$ can
  be disconnected by at most
  $\lfloor\frac{k-1}{2}\rfloor+ \lfloor\frac{k}{2}\rfloor = k-1$
  edges, contradicting the $k$--edge-connectedness of $G_1$.

  Let $E$ be a set of $(k-1)$ or fewer edges in $G$.  Either $G_1'$
  or $G_2'$ contains at most $\lfloor\frac{k-1}{2}\rfloor$ edges in
  $E$.  Switching labels if necessary, we suppose it is $G_1'$.  Since
  $G_1'$ is $(\lfloor\frac{k-1}{2}\rfloor+1)$--edge-connected, the subgraph
  $G_1'\smallsetminus E\subset G\smallsetminus E$ is connected.
  
  Let $x_1,\ldots,x_k$ be the (not necessarily distinct) vertices of
  $G_1$ connected to $v_1$, and let $E_0$ be the set of $k$ edges in
  $G$ connecting $G_1'$ to $G_2'$.
  Finally, let $G_2''\subset G$ be
  $$G_2'\cup\{x_1,\ldots,x_k\}\cup E_0.$$   Notice that there is a
  quotient map $p\co G_2''\onto G_2$ which takes each $x_j$ to $v_2$.
  Since $G_2$ is $k$--edge-connected, $p(E\cap G_2'')$ doesn't
  disconnect $G_2$.  This implies that every vertex of $G_2''$ is
  connected in $G_2''\smallsetminus E$ to $x_j$ for some $j$.  Since
  the $x_j$ are all in $G_1'$, and $G_1'\smallsetminus E$ is
  connected, $G\smallsetminus E$ is connected.
\end{proof}

We now explain how splicing arises in building Whitehead graphs.
Let $\underw$ be a collection of words, as in Section \ref{s:whitehead}.
Let $\mathcal{C}$ be a complete
system of disks for the handlebody $H$, and let $\tilde{H}$ be some
$d$--fold cover of $H$, with $\pi_1(\tilde{H}) = F'$.
The preimage $\tilde{\mathcal{C}}$ in
$\tilde{H}$ has too many
disks to be a complete disk system for $\tilde{H}$, but we can still
form the graph $W(\underw_{F'},\tilde{\mathcal{C}})$.  Each component
of the complement of $\tilde{\mathcal{C}}$ looks exactly like the
complement of $\mathcal{C}$ in $H$, so we obtain the following:
\begin{lemma}\label{l:disjoint}
  The graph $W(\underw_{F'},\tilde{\mathcal{C}})$ is isomorphic to a
  disjoint union of $d$ copies of $W(\underw,\mathcal{C})$.
\end{lemma}
To get a complete system of disks for $\tilde{H}$, we must delete some
disks from $\tilde{\mathcal{C}}$.   Each time we delete one of these
disks, the number of complementary components goes down by one, and
the Whitehead graph changes by splicing two components 
together:
\begin{lemma}\label{l:splicecomplete}
  Let $\mathcal{D}_0\subseteq\mathcal{D}_1\subset\mathcal{D}_2$ be
  disk systems in a handlebody $H$ so that $\mathcal{D}_0$ is complete
  and $\mathcal{D}_2\smallsetminus\mathcal{D}_1$ is a single disk.
  Let $\underw$ be a collection of conjugacy classes in $\pi_1(H)$.  
  The graph $W_1 = W(\underw,\mathcal{D}_1)$ is obtained from $W_2 =
  W(\underw,\mathcal{D}_2)$ by splicing two components of $W_2$
  together at a vertex.
\end{lemma}
(Similarly, the Whitehead move in Figure \ref{f:whmove}
can be described as an ``unsplicing'' of
the edges $e_1,\ldots,e_s$ followed by 
splicing the resulting components together at $v_\pm$.)

\section{Examples}\label{s:answer}
\begin{theorem}\label{t:main}
  Let $\mathcal{C}$ be a complete system of disks for 
  the handlebody $H$, and let
  $\underw$ be a collection of conjugacy classes in $\pi_1(H)$.  Let
  $k\geq 3$.
  If $W(\underw,\mathcal{C})$ is a regular, $k$--valent,
  $k$--edge-connected non-planar graph, then $\underw$ is not virtually
  geometric. 
\end{theorem}
\begin{proof}
  Let $\tilde{H}$ be a $d$--fold cover of $H$, corresponding to $F'<F$,
  where $F = \pi_1(H)$, and let
  $\tilde{\mathcal{C}}$ be the preimage of $\mathcal{C}$ in
  $\tilde{H}$.  

  By Lemma \ref{l:disjoint}, the Whitehead graph
  $W(\underw_{F'},\tilde{\mathcal{C}})$ is equal to $d$ disjoint copies
  of $W(\underw,\mathcal{C})$.  If $\mathcal{D}\subset
  \tilde{\mathcal{C}}$ is a complete system of disks for $\tilde{H}$,
  then Lemma \ref{l:splicecomplete} implies that
  $W(\underw_{F'},\mathcal{D})$ is obtained from
  $W(\underw_{F'},\tilde{\mathcal{C}})$ by successively splicing together
  distinct components in some way (determined by the choice of
  complete system
  $\mathcal{D}\subset \tilde{\mathcal{C}}$) until only
  $2\cdot \mathrm{rank}(F')$ vertices remain.  Exactly $d-1$ splices
  will be performed, and each splice reduces
  the number of components by one, by Lemma \ref{l:connected}, so
  $W(\underw_{F'},\mathcal{D})$ will be connected.

  Moreover, induction on $d$ together with  Lemma \ref{l:connected} 
  imply that
  $W(\underw_{F'},\mathcal{D})$ is $k$--valent and
  $k$--edge-connected.
  By Theorem \ref{t:reducible}, the disk system $\mathcal{D}$ is
  minimal.  

  A $k$--valent, $k$--edge-connected graph cannot be disconnected by
  removing a single vertex, so induction on $d$ and
  Lemma \ref{l:minor} imply that the graph $W(\underw,\mathcal{C})$ is a
  minor of $W(\underw_{F'},\mathcal{D})$.  
  It follows that $W(\underw_{F'},\mathcal{D})$ is
  non-planar.  

  Theorem \ref{t:planar} then implies that $\underw_{F'}$ is
  not geometric in $\tilde{H}$.  Since the finite cover $\tilde{H}$
  was arbitrary, $\underw$ is not virtually geometric.
\end{proof}
A concrete example is the word $bbaaccabc\in F_3=\langle a,b,c\rangle$; the Whitehead graph
$W(\{bbaaccabc\})$ is isomorphic to the complete bipartite graph
$K_{3,3}$.  This graph is  $3$--valent, $3$--edge-connected, and
non-planar.  Applying Theorem \ref{t:main} gives the following.
\begin{corollary}\label{cor:k33}
  The word $w=bbaaccabc$ in $F_3 = \langle a,b,c\rangle$ is not
  virtually geometric.
\end{corollary}
For an example lying in the commutator subgroup, we take $w =
baabccACBBCA$.  (We use the common convention $A=a^{-1}$, $B=b^{-1}$,
$C=c^{-1}$.) 
The Whitehead graph, shown in Figure \ref{f:comm}, is
\begin{figure}
  \input{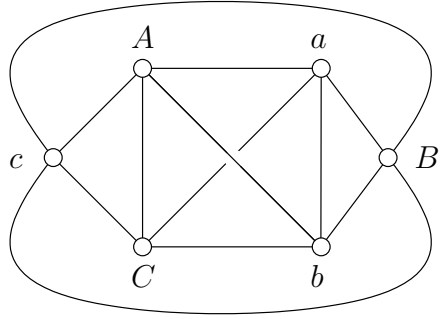}
  \caption{The Whitehead graph of the word $w = baabccACBBCA$ is
    non-planar, $4$--valent, and $4$--edge-connected.}
  \label{f:comm}
\end{figure}
non-planar, $4$--valent, and $4$--edge-connected, so Theorem
\ref{t:main} immediately gives the following.
\begin{corollary}\label{cor:comm}
  The word $w= baabccACBBCA$ in $F_3 = \langle a,b,c\rangle$ is not
  virtually geometric.
\end{corollary}

The curious reader should be able to construct numerous similar
examples, but it should be noted that either of the above examples gives
non-virtually geometric words in $F_n$ for all $n\geq 2$, as the
following proposition shows.

\begin{proposition}
Let $\underw = \{w\}$ be a non-virtually geometric collection 
of words in $F_3$.
For all $n\geq 2$ there is an embedding $\phi_n\co F_3\to F_n$ so that
$\phi_n(\underw)$ is not virtually geometric in $F_n$.
\end{proposition}
\begin{proof}
For $n=2$ we argue as follows:
Let $\phi_2\co F_3\to F_2$ be any
map with image an index $2$ subgroup.    Indeed, the homomorphism $\phi_2$ is
realized by a double cover of handlebodies $H_3\to H_2$, and any cover 
$\tilde{H}$ of
$H_2$ either covers $H_3$ itself or has a double cover which covers $H_3$.
The lifts of $\underw$
in this cover are contained in the lifts of $\phi_2(\underw)$.
Thus if $\phi_2(\underw)$ were virtually
geometric, then $\underw$ would be virtually geometric.

For $n>3$, we let $\phi_n\co F_3\to F_n$ be any embedding of $F_3$ as
a free factor of $F_n$.
To see why $\phi_n(\underw)$ is non-virtually geometric, it is easiest to argue topologically.
Let $H_3$ be the handlebody of genus $3$; we can obtain a handlebody $H_n$
of genus $n$ by attaching $n-3$ one-handles along $2n-6$ disjoint
disks in the boundary of $H_3$.  
Realize $\underw$ as an embedded
$1$--manifold $N$ in $H_3\subset H_n$.  Suppose there is some finite cover
$\tilde{H}\stackrel{p}{\to} H_n$ in which the preimage of $N$ is
homotopic to an embedded submanifold of the boundary.  
There is a corresponding finite cover $\tilde{H}'\stackrel{p'}{\to}H_3$
so $\tilde{H}'$ embeds in $\tilde{H}$.  There is also an infinite-sheeted cover $\tilde{H}''\stackrel{p''}{\to}\tilde{H}$ corresponding
to $\pi_1(\tilde{H}')\subset\pi_1(\tilde{H})$; the space $\tilde{H}''$ can be obtained from $\tilde{H}'$ by attaching
finitely many non-compact $3$--manifolds with boundary $A_1,\ldots,A_l$
to $\tilde{H}'$
along disks.  Each of the $A_i$ is homeomorphic to a ball with a
Cantor set removed from its boundary.  
\cd{
 & \tilde{H}''\ar@{=}[r]\ar[d]^{p''} & \tilde{H}'\cup\bigcup_iA_i\\
\tilde{H}'\ar[ur]\ar[r]\ar[d]^{p'} & \tilde{H}\ar[d]^p & \\
H_3 \ar[r] & H_n &
}%
Notice that 
${p'}^{-1}(N)\subseteq {(p \circ p'')}^{-1}(N)$ coincides with the union of the
compact components of ${(p \circ p'')}^{-1}(N)$.   If $p^{-1}(N)$ is
homotopic to an embedding in the boundary of $\tilde{H}$, then
${(p \circ p'')}^{-1}(N)$ is homotopic to an embedding in the boundary of
$\tilde{H''}$.  Since $\tilde{H}'' = \tilde{H}'\cup\bigcup_iA_i$, this embedding is
homotopic to an embedding in the boundary of $\tilde{H}'$, which
covers $H_3$.  This
contradicts the assumption that $\underw$ is not virtually
geometric in $F_3$.
\end{proof}
\begin{remark}
  Gordon and Wilton's motivation for asking Question~\ref{q:main} was
  the question of which doubles of free groups contain closed surface
  subgroups.  The double $D_3(\{baabccACBBCA\})$ contains a
  surface subgroup, by an application of the main theorem of
  \cite{calegari:surface}.  The double $D_3(\{bbaaccabc\})$ contains a
  closed surface subgroup by an explicit construction of Sang-hyun Kim
  \cite{Kim09}.
\end{remark}

\section{Acknowledgments}
The author thanks Danny Calegari for pointing out this problem,
Danny Calegari, Daryl Cooper, and Richard Schwartz for
useful conversations, and 
Nathan Dunfield for a python script which was useful for running
computer experiments with John Berge's program Heegaard.  Thanks also to
John Berge for pointing out a bad typo in an earlier version of this
note, and thanks to the anonymous referee for useful comments. 

The author was visiting 
the Caltech mathematics department
while this work was done, and
thanks Caltech for their hospitality.
This work was partly supported by the National Science Foundation,
grant DMS-0804369. 

\end{document}